\documentclass[11pt]{amsart}
\usepackage{amssymb,amsmath,amsthm,amscd}
\usepackage[all]{xy}
\numberwithin{equation}{section}        
\newtheoremstyle{fancy}{}{}{\itshape}{}{\textsc\bgroup}{.\egroup}{ }{}
\theoremstyle{fancy}
\newtheorem{corollary}[equation]{Corollary}     
\newtheorem{lemma}[equation]{Lemma}
\newtheorem{proposition}[equation]{Proposition}
\newtheorem{theorem}[equation]{Theorem}
\newtheorem{problem}[equation]{Problem}
 
\newtheorem*{conjecture}{Conjecture}

\theoremstyle{definition}

\theoremstyle{remark}
\newtheorem{example}[equation]{Example}
\newtheorem{remark}[equation]{Remark}

\newcommand{\cref}[1]{Corollary~\ref{#1}}

\newcommand{\no}{\noindent}

\newcommand{\fg}{{\mathfrak{g}}}

\newcommand{\fh}{{\mathfrak{h}}}

\newcommand{\cS}{{}_cS }
\newcommand{\cM}{{}_cM }
\newcommand{\cG}{{}_cG }
\newcommand{\rM}{{}^r\!M }
\newcommand{\resM}{{}_cM \times _{{}_cG} G/H }
\def\con#1=#2(#3){#1 \equiv #2 \bmod{#3}}

\begin{document}

\title{Global $G$-Manifold Reductions and Resolutions}

\dedicatory{Dedicated to the memory of Alfred Gray}

\author{Karsten Grove$^{*}$}
\address{University of Maryland\\
 College Park , MD 20742}
\email{kng@math.umd.edu}
\author{Catherine Searle$^{**}$}
\address{Instituto de Matematicas - UNAM\\
Cuernavaca, Mexico}
\email{csearle@matcuer.unam.mx}
\thanks{$^*$ Supported in part by a grant from the National
Science Foundation and by the Danish National Research Council.}
\thanks{$^{**}$ Supported in part by CONaCYT grant no, 28491E.}

\begin{abstract}
The purpose of this note is to exhibit some simple and basic constructions for smooth compact transformation groups, and some of their most immediate applications to geometry. 
\end{abstract}

\maketitle

It is well known that, if
$G$ is a compact Lie group acting smoothly on a manifold $M$ with only one orbit type,
then the orbit space
$M/G$ is a manifold, and the orbit map $\pi : M \to M/G$ is a locally trivial bundle
with fiber $G/H$, the typical $G$-orbit in $M$. Moreover, the normalizer
$N(H)$ acts on the fixed point set, $M^H \subset M$ of $H$  in $M$, with $H$ as the
ineffective kernel, and $M^H \to M^H/N(H) = M/G$ is a principal
$N(H)/H$-bundle. Its associated $G/H$-bundle is $\pi$. In particular, if
we set $\cM = M^H$ and $\cG = N(H)/H$ then $M = \resM$. This shows that we
can recover $M$ as a $G$-manifold completely from the $\cG$-manifold $\cM$ and
$H
\subset G$. We will refer to $\cM$ as the \textit{core} of $(M,G)$ and to
$\cG$ as its \textit{core group}. 

In general, when the action $G \times M \to M$ has more than one orbit type
no such simplification exists. It turns out, however, that if $H$ is a
principal isotropy group, then the closure, $\cM$ of the core ${}_c(M_o) =
(M_o)^H$ of the regular part $M_o$ of $M$ (all principal orbits) is a smooth
$\cG$-manifold which contains considerable information about $(M,G)$ (see Proposition \ref{p.7}
and Proposition \ref{p.9}). This was already indicated in \cite{BH} and used in \cite{S}. The
unpublished manuscript, \cite{SS} by T. Skjeldbred and E. Straume is devoted
to the basic investigation of this core, $\cM$ of $(M,G)$ referred to as the
\textit{reduction} by them. Given its importance, we have included complete
proofs of this basic material. Our proofs are different from those of
\cite{SS}, in that we use Riemannian geometric tools from the outset.

We use the core to obtain restrictions on positively curved $G$-manifolds. In
particular, we obtain an extension of a fixed point lemma (cf. Corollary \ref{c:3}), used in our
systematic investigation of symmetry groups of positively curved manifolds
in \cite{GS}.

As another application of the core, we associate to any smooth $G$-manifold,
$M$ with $G$ a compact Lie group another $G$-manifold $\rM$ and a smooth
surjective $G$-map $f\!:\rM \to M$. The orbit space of $\rM$ is the same as
that of $M$, but $\rM$ is less singular than $M$ in the sense that
corresponding orbits have smaller isotropy groups (Theorem \ref{2.5}). For this
reason, we think of $\rM$ as a (partial) \textit {resolution} of $M$. In
contrast to other regularizations of group actions, e.g., \textit {blow-up}
along invariant submanifolds as in \cite{Was}, our construction is completely global in
nature. In terms of the core, $\rM = \resM$ fibers over $G/N(H)$ with fiber $G/H$. - A
geometric feature of the resolution construction is that it preserved the class of
manifolds with non-negative curvature (2.10). In particular, new manifolds of
non-negative curvature may possible be constructed by this method. We also point out
natural problems and conjectures related to the constructions in this note.

\section{The core of a $G$-manifold}

Since we only consider smooth compact transformation groups, throughout we
may as well assume that each transformation is an isometry relative to a
fixed (auxiliary) Riemannian metric.

We first recall some well known facts (cf. e.g. \cite{Br}) and establish
notation. - Throughout, $M$ will denote a closed, connected Riemannian
manifold, and $G$ a compact Lie group which acts isometrically and
effectively on $M$. For $p \in M$, $G_p = \{g \in G \mid gp = p\}$ is the
\textit{isotropy group} of $G$ at $p$, and $Gp = \{gp \mid g \in G\} \simeq
G/G_p$ is the orbit through $p$. For any (closed) subgroup $L \subset G$,
$M^L = \{p \in M \mid Lp = p\}$ will denote the fixed point set of $L$ in $M$.
$M^L$ is a finite union of closed totally geodesic submanifolds of $M$.

We endow the orbit space, $M/G$ with the so-called \textit{orbital metric}, i.e., if
$\pi: M \to M/G$ is the quotient map, then the distance between $\pi(p)$ and
$\pi(q)$ is the Riemannian distance in $M$ between the orbits $Gp$ and $Gq$.
With this metric, $\pi$ is a \textit{submetry}, i.e., for any $p \in M$ and
any $r>0$, the $r$-ball around $p$, $B(p,r)$ is mapped onto the $r$-ball,
$B(\pi(p),r)$ around $\pi(p)$. - On the regular part $M_o \subset M$
consisting of all principal orbits, the restriction $\pi:M_o \to M_o/G$ is a
Riemannian submersion and a locally trivial bundle map with fiber $G/H$,
where $H = G_{p_o}$ for some $p_o \in M_o$. If the principal isotropy type
$(H)$ is trivial, the bundle $\pi:M_o \to M_o/G$ is a principal $G$-bundle.

Fix a principal isotropy group $H$, and assume from now on that $H \ne
\{1\}$. The normalizer, $N(H)$ of $H$ in $G$ clearly acts on $M_o^H = M_o
\cap M^H$ with $H$ as ineffective kernel. Moreover, each principal
$G$-orbit, $G/H$ intersects $M_o^H$ in an $N(H)$-orbit, $(G/H)^H \simeq
N(H)/H =: \cG$, and the inclusion $M_o^H\subset M_o$ induces an isometry
$M_o^H/\cG \simeq M_o/G$. The induced action of the \textit{core group}, $\cG$
on $M_o^H$ is free, and
\begin{equation*}
\tag{1.1}
M_o^H \to M_o^H/\cG \simeq M_o/G
\end{equation*}

\no is the principal bundle for $M_o \to M_o/G$. Indeed, the $G$-action
\begin{equation*}
\tag{1.2}
G \times M_o^H \times G/H \to M_o^H \times G/H ,   (g,(x,[g'])) \to (x, [gg'])
\end{equation*}
\no where $[g] = gH$, commutes with the $\cG$-action
\begin{equation*}
\tag{1.3}
M_o^H \times G/H \times \cG \to M_o^H \times G/H,   ((x, [g']), [n]) \to
(n^{-1}x,[g'n])
\end{equation*}
\no where $[n] = nH \in N(H)/H = \cG$. Moreover, the $G$-map
\begin{equation*}
\tag{1.4}
M_o^H \times G/H \to M_o ,   (x,[g']) \to g'x
\end{equation*}
\no induces a $G$-diffeomorphism
\begin{equation*}
\tag{1.5}
M_o^H \times _{{}_cG}G/H := (M_o^H \times G/H)/{}_cG \to M_o
\end{equation*}
\no identifying $M_o \to M_o/G$ with the $G/H$-bundle associated with the
principal bundle (1.1).

The construction of the core of $M$ (called the reduction of $M$ in \cite
{SS}) and later of the (partial) resolution of $M$, are natural extensions
of the well known facts outlined above. 

We refer to the closure, $\cM$ := cl$(M_o^H)$ of $M_o^H$ in $M$, as the
\textit{core} of $M$. Clearly, each of the sets  $M_o^H \subset \cM \subset
M^H$ are invariant under the $\cG$-action, and in general each inclusion is
strict (cf. e.g. \cite {SS}, for an example where $\cM \ne M^H)$. Our first
objective is to analyze the structure of $\cM$.

The following possibly well known simple but very useful fact can be found in
Kleiners thesis \cite{K}.

\begin{lemma}
\label{l.6}
Let $c:[0,\ell] \to M$ be a minimal geodesic between the orbits $Gc(0)$ and
$Gc(\ell)$. Then $G_{c(t)} = G_{c}$ for all $t\in(0,\ell)$ and $G_{c(0)}
 \supset G_{c} \subset G_{c(\ell)}$.
\end{lemma}

This together with the slice theorem can be used to give simple geometric
proofs of all basic facts about compact transformation groups. Here we will
use it in the proof of

\begin{proposition}
\label{p.7}
The core $\cM$ is a smooth submanifold of $M$. In fact, $\cM$ is the disjoint
union of those components, $F$ of $M^H$ such that $F \cap M_o \ne \emptyset$.
(For these componets, dim$F$ = dim$\cG$ + dim$M/G$.)
\end{proposition}

\begin{proof}
Since $(G/H)^H \simeq \cG$ and hence $M_o^H$ has only finitely many
components (all diffeomorphic), the inclusion $\cM \subset \cup F$, $F$
component of $M^H$ with $F \cap M_o^H \ne \emptyset$, is obvious.
Now let $x \in F$, with $F$ as above. Choose $y \in$ cl$(F \cap M_o^H)
\subset \cM$ closest within $F$ to $x$. We will show that $x=y$. First note
that since $F \cap M_o^H$ is open in $F$ and contains components of
$\cG$-orbits in $M_o^H$ arbitrarily close to $y$, we can find unit tangent
vectors $v \in T_yF$ such that by Lemma \ref{l.6}, exp$(tu) \in F \cap M_o^H$
for all $u$ close to $v$ and all small positive $t$, and hence all small $t
\ne 0$. Now suppose $x \ne y$ and let $\gamma$ be a minimal geodesic in $F$
from $y$ to $x$. By assumption, all points of $\gamma$ except $y$ are in
$F-M_o^H$. However, from the above, there are minimal geodesics $c$ in $F$
emanating from $y$ all of whose points except $y$ are in $M_o^H$, and such
that $c$ makes an angle less than $\pi/2$ with $\gamma$. This contradicts the
choice of $y$, and hence $x=y$.
\end{proof}

Our next goal is to determine the regular part of the core action
$\cG\times\cM \to \cM$ and its orbit space. We need the following

\begin{lemma}
\label{l:8}
Suppose $G$ acts isometrically on the unit $n$-sphere, $S$ with principal
isotropy group $H$. Then $N(H)/H \ne\{1\}$.
\end{lemma}

\begin{proof}
Assume without loss of generality that $S^G = \emptyset$ (otherwise look at
$(S^G)^\perp$). A simple convexity argument shows that in this case diam$S/G
\le \pi/2$. If $H = \{1\}$, there is nothing to show. Now suppose $H \ne
\{1\}$ and $N(H)/H = \cG = \{1\}$. If $G$ acts transitively, $S = G/H$ we have
$S^H
\simeq (G/H)^H \simeq cG$ which is impossible since $S^H$ is a subsphere. If
$G$ does not act transitively, consider $\cS \supset (S_o)^H$ which in this
case is a connected subsphere of $S$. The assumption $\cG = \{1\}$ implies
that dist$(x,y)$ = dist$({}_cGx,{}_cGy)$ = dist$(Gx,Gy)$ for all $x,y \in
(S_o)^H$. Since $(S_o)^H$ is dense in $\cS$, and $(S_o)^H \simeq (S_o)^H/\cG
\simeq S_o/G$ is dense in $S/G$ we get a contradiction from diam$\cS = \pi$ and
diam$S/G \le \pi/2$.
\end{proof}

We are now ready to prove

\begin{proposition}
\label{p.9}
The inclusion $\cM \subset M$ induces an isometry $\cM/\cG \simeq M/G$ and
$(\cM)_o = (M_o)^H$. In particular, $\cG$$x = Gx \cap \cM$ and $(\cG)_x =
{}_c(G_x)$ for all $x \in \cM$.
\end{proposition}

\begin{proof}
Clearly $(M_o)^H \subset (\cM)_o$. To prove the opposite inclusion let $x \in
\cM - (M_o)^H$. We need to see that $(\cG)_x \ne \{1\}$. However, $(\cG)_x =
(N(H) \cap G_x)/H$, and $G_x$ acts on the normal sphere, $S_x^\perp$ to the
orbit $Gx$ with principal isotropy group $H$. Thus $(\cG)_x$ can also be
viewed as the core group ${}_c(G_x)$ for the $G_x$-action on $S_x^\perp$ and
the claim follows from Lemma \ref{l:8}.
Since $\cM \subset M \to M/G$ is clearly surjective, and$(\cM)_o/\cG =
(M_o)^H/\cG \simeq M_o/G$ is an isometry, the extension $\cM/\cG \to M/G$ is
an isometry as well.
\end{proof}

The following is a simple (and possibly well known) observation based on the
construction above.

\begin{theorem}
\label{t:0}
Let $M$ be a $G$-manifold with principal isotropy group $H$. If the core
group $\cG = N(H)/H$ is trivial, then all orbits are principal and $M$ is
$G$-equivalent to $\cM \times G/H$.
\end{theorem}

\begin{proof}
Since $\cG = \{1\}$, $(\cM)_o = \cM$. But then $M/G \simeq \cM/\cG = (\cM)_o/
\cG = (M_o)^H/\cG = M_o/G$, i.e., $M_o = M$. Moreover, $M = M_o \to M_o/G =
M/G$ is a bundle with fiber $G/H$ and trivial principal bundle $M^H = (M_o)^H
\overset{\sim}{\rightarrow} (M_o)^H/\cG = M^H \simeq M/G$.
\end{proof}

This suggests an extension of Lemma \ref{l:8} to manifolds of positive curvature:

\begin{theorem}
\label{t:1}
Let $M$ be a closed manifold of positive curvature and $G$ a compact group of
isometries on $M$. If $H \subset G$ is a principal isotropy group, then
$N(H)/H \ne \{1\}$ unless $M = G/H$. 
\end{theorem}

\begin{proof}
If $N(H)/H = \{1\}$ we know from Theorem \ref{t:0} that $M$ is $G$-equivalent to
$\cM \times G/H$. Moreover, the projection $M \simeq \cM \times G/H \to \cM
\simeq M/G$ is a flat Riemannian submersion, i.e., has trivial integrability
tensor ($G$-translates of $\cM$ are integral submanifolds of the horizontal
distribution). Since a flat Riemannian foliation in a non-negatively curved
manifold locally splits isometrically by \cite[Theorem 1.3]{Wa}, this is
impossible in positive curvature unless $\cM$ is a point, i.e., $M = G/H$ is
homogeneous. 
\end{proof}

For calculations of $N(H)/H)$ when $G$ is $1$-connected, see \cite{Sh}, where it is called the generalized Weyl group.
The description of $G$-manifolds whose core group $\cG = N(H)/H$ is finite is
considerably more complicated than Theorem \ref {t:1}
. These are, however, special
cases of so-called polar manifolds, see \cite{PT}. For a detailed analysis of polar
manifolds we refer to \cite {GZ}. Here, however, we point out that the
arguments of Theorem \ref {t:1} above can be pushed to yield the following
extension:

\begin{theorem}
\label{t:2}
Let $M$ be a closed manifold of positive curvature, and $G$ a compact Lie
group of isometries on $M$, with principal isotropy group $H \subset G$. If
the core group $\cG = N(H)/H$ is finite then either 

\no $(a)$ $M = G/H$ , or

\no $(b)$ There are singular orbits $Gx$ on $M$, i.e.., there are points $x \in
M - M_o$ with dim$G_x >$ dim$H$.
\end{theorem}

\begin{proof}
Suppose dim$G_x =$ dim$H$ for all $x \in M$ and that $\cG$ is finite. The
first assumption implies that the $G$-orbits on $M$ define a Riemannian
foliation of $M$. The second assumption implies that this foliation is flat,
i.e., has integrable horizontal distribution as in the proof of Theorem \ref{t:1}.
As in Theorem \ref {t:1} we get from \cite{Wa} that $\cM$ must be a point if $M$
has positive curvature, hence $M = G/H$.
\end{proof}

The following immediate corollary played an important role in \cite[Theorem
B]{GS}.

\begin{corollary}[Fixed point lemma]
\label{c:3}
Let $M$ be a closed manifold of positive curvature and $G$ a compact
connected Lie group of isometries on $M$. If the identity component $H_o$ of
the principal isotropy group $H$ is a non-trivial maximal connected subgroup
of $G$, then either

\no $(a)$ $M = G/H$ , or

\no $(b)$  $M^G \ne \emptyset$.
\end{corollary}

\begin{proof}
Since $H_o$ is maximal we see that $N(H)/H$ is finite. Otherwise, dim$N(H)$ =
dim$G$ and hence $M^H = M$.
\end{proof}

We conclude this section by pointing out that part of the essence of the core
group and manifold is, that it reduces many general questions about group
actions to those that have trivial principal isotropy group. Some of the core
constructions generalize to other types of "reductions" when replacing
$(M_o)^H$ by $(M_o)^L$ for subgroups $L \subset H$. Among all these
reductions, the core, $\cM$ is the one reduced the most. This is the reason
for choosing the word "core" for this most basic reduction. We will not
pursue the more general reductions further in this note.

\section{The core-resolution construction}

With the construction of $M_o$ in (1.5) as guideline, we will construct
a new $G$-manifold, $\cM$ which maps onto $M$ but is less singular as a
$G$-manifold, i.e., has smaller isotropy groups.

The $G$-action on $M$ induces a smooth surjective map
\begin{equation*}
\tag{2.1}
F: \cM \times G/H \to M , (x,[g']) \to g'x
\end{equation*}

\no extending the map in (1.4). This is $G$-equivariant when $G$ acts
trivially on the $\cM$-factor, and by left translations on $G/H$. Moreover,
it is $\cG$-invariant relative to the obvious $\cG$-extension to $\cM \times
G/H$ of (1.3). Thus $F$ induces a surjective $G$-equivariant map
\begin{equation*}
\tag{2.2}
f: (\cM \times G/H)/\cG \to M , \cG(x,[g']) \to g'x.
\end{equation*}

\no Since the natural $\cG = N(H)/H$-action
\begin{equation*}
\tag{2.3}
G/H \times \cG \to G/H , ([g'],[n]) \to [g'n]
\end{equation*}

\no is free, $\rM := (\cM \times G/H)/\cG = G/H \times _{\cG}\cM$ is a
smooth manifold. We will refer to $\rM$ as the (core-) \textit{resolution} of $M$.
Note that we can view $\rM$ as a bundle over $(G/H)/\cG = G/N(H)$ with
fiber $\cM$ associated to the principal $\cG$-bundle $G/H \to G/N(H)$. This
bundle map
\begin{equation*}
\tag{2.4}
\rM \to G/N(H) , (x,[g']) \cG \to [g'] \cG =:(g'),
\end{equation*}

\no where $(g') = g'N(H)$, is clearly $G$-equivariant.
We will now analyze the $G$-manifold $\rM$ and the map $f$ in \eqref{2.2} in
more detail.

\begin{theorem}
\label{2.5}
Let $M$ be a $G$-manifold with resolution $f: \rM \to M$ as in \eqref{2.2}, \\and 
$F: \cM \times G/H \to M$ as in (2.1). Then

\no $(1)$ $f: \rM_o \to M_o$ is a $G$-diffeomorphism

\no $(2)$ $f$ restricts to a $\cG$-diffeomorphism between the cores of $\rM$ and
of $M$

\no $(3)$ $f/G: \rM/G \to M/G$ is a homeomorphism.

\no $(4)$ $G_{(x,[1])\cG} = N(H) \cap G_x = N^{G_x}(H), (x,[1]) \in \cM
\times G/H$.

\no $(5)$ $DF_{(x,[1])}: T_x(\cM) \times T_{[1]}G/H \to T_xM$ is surjective if
and only if (*) $T_x(\cM) + T_xGx = T_xM$.

\no $(6)$ The condition (*) is equivalent to $G_x \subset N(H)$ as well as to
$(G_x)_o = (N(H) \cap G_x)_o$,
and $f: \rM \to M$ is a $G$-diffeomorphism if this condition holds for all $x
\in \cM$.
\end{theorem}

\begin{proof}
We prove (4) first: Let $x \in \cM \subset M$. Then
\begin{align*}
G_{(x,[1])\cG} &= \{g \mid g(x,[1])\cG = (x,[g])\cG = (x,[1])\cG\} \\
&= \{g \mid \exists n \in N(H): (x,[g]) = (n^{-1}x,[n])\} \\
&= \{g \mid \exists n \in N(H) \cap G_x: gH = nH\} \\
&= \{g \mid \exists n \in N(H) \cap G_x: g \in nH\} \\
&= N(H) \cap G_x = N ^{G_x}(H).
\end{align*}

\no This together with Lemma \ref{l:8} (cf. also Proposition \ref{p.9}) shows that for $x \in
\cM$, $(x,[1])\cG \in \rM_o$ if and only if $x \in \cM_o = M_o^H$. In other
words $\rM_o = (M_o^H \times G/H)/\cG \simeq M_o$.  As for (2) first note that
$(\rM_o)^H \simeq (M_o^H \times G/H)/\cG \simeq M_o^H$ and cl$(\rM_o)^H = (\cM
\times N(H)/H)/\cG \simeq \cM$. In particular, $f$ restricts to a
$\cG$-diffeomorphism from the core of $\rM$ to the core of $M$.  We get (3)
as an immediate consequence of (2) and Proposition \ref{p.9}. Now, for $x \in \cM$ and
$v \in T_x(\cM)$, $DF_{(x,,[1])}(v,0) = v \in T_x(\cM) \subset T_xM$. Moreover,
for $X \in T_{[1]}G/H \simeq \fh^{\perp}$ we have $DF_{(x,[1])}(0,X) = X^*(x)$,
where $X^*$ denotes the action field on $M$ corresponding to $X \in \fg \simeq
T_1G$. In particular, $DF(T_{(x,[1])}\cM \times G/H) = T_x(\cM) + T_xGx$ and (5)
is proved.  Since $T_x(\cM) \cap T_xGx = (T_xGx)^H$ and its complement in
$T_xGx$ is perpendicular to $T_x(\cM)$, we see that the condition (*) in (5) is
equivalent to the condition $T_xGx^{\perp} \subset T_x(\cM)$. This on the other
hand says that $H$ acts trivially on $T_xGx^{\perp}$, i.e., $H$ is normal in
$G_x$. The condition (*) is also equivalent to the condition 
$$ \text{dim}\cM + (\text{dim}Gx - \text{dim}{}_cGx) = \text{dim}M $$  
by Proposition \ref{p.9}. The left hand side can be written as \\ \\
\indent{dim$M_o^H$ + (dim$G$ - dim$G_x$ - dim$\cG$ + dim$\cG_x$) =} \\
\indent{dim$M/G$ + dim$\cG$ + (dim$G$ - dim$G_x$ - dim$\cG$ + dim$G_x \cap N(H)$ -
dim$H$) =} \\
\indent{dim$M$ + dim$G_x \cap N(H)$ - dim$G_x$, and dim$G_x \cap N(H)$ = dim$G_x$} \\ \\
if and
only if
$(G_x)_o = (G_x \cap N(H))_o$. By $G$-equivariance, $F$ and hence $f$ is a
submersion if and only if $DF_{(x,[1])}$ is surjective for all $x \in \cM$. In
this case, $f$ is a covering map and thus a diffeomorphism by (1). This
completes the proof of (6).
\end{proof}

\begin{remark}
\label{2.6}
From this theorem we see in particular that $\rM/G = M/G$. Moreover, $\rM = M$
if and only if $G_x \subset N(H)$ (or equivalently $(G_x)_o = (N(H) \cap G_x)_o)$
for all $x \in \cM$ and hence ${}^r(\rM) = \rM$. Among all $G$-manifolds with
the same principal isotropy group and the same core $(\cM, \cG)$, their common
resolution is the least singular $G$-manifold. Also if $M_1$ and $M_2$ are two
$n$-dimensional $G$-manifolds with $\rM_1 = \rM_2$ as $G$-manifolds, then $\cM_1
= \cM_2$ as $\cG$-manifolds. Note however, that $(\cM_1, \cG) = (\cM_2, \cG)$
does not imply that $(\rM_1, G) = (\rM_2, G)$ as the following example shows.
\end{remark}

\begin{example}
\label{2.7}
If $M = G/H$ is homogeneous, clearly $\rM = M$. In this case, the core $\cM =
(G/H)^H = N(H)/H \simeq \cG$ acting through right translations on itself. In
particular if $M_i = G/H_i, i = 1,2$ and $N(H_1)/H_1 = N(H_2)/H_2$ then $(\cM_1,
\cG) = (\cM_2, \cG)$ but $\rM_1 = M_1 \ne M_2 = \rM_2$ if $G/H_1 \ne G/H_2$. To
be explicit, such examples can be found among the Aloff-Wallach examples
\cite{AW}, $M_{p,q} = SU(3)/S^1_{p,q}$. When $(p,q) \ne (1,1)$ and gcd$(p,q) =1$,
then $N(S^1_{p,q})/S^1_{p,q} = S^1$ (cf. e.g. \cite{Sh}).
\end{example}

\begin{corollary}
\label{2.8}
If $M$ is a $G$-manifold without singular orbits then $\rM = M$.
\end{corollary}

\begin{proof}
By assumption dim$G_x$ = dim$H$ for all $x \in M$. But for $x \in \cM$ we have
$H \subset N(H) \cap G_x$ and hence $H_o \subset (N(H) \cap G_x)_o \subset
N(H)_o \cap (G_x)_o \subset (G_x)_o = H_o$, and the claim follows from
Theorem \eqref{2.5}.
\end{proof}

For the following cf. also  the example $(M,G) = (S^n,SO(n))$
where $S^n =$ \\
$ \Sigma S^{n-1} = \Sigma SO(n)/SO(n-1)$, and Theorem \ref{t:0}.

\begin{corollary}
\label{2.9}
Let $M$ be a 1-connected $G$-manifold with principal isotropy group $H$. If
there are no singular orbits in $M$, and $N(H)/H$ is finite, then all orbits are
principal.
\end{corollary}

\begin{proof}
From Corollary \ref{2.8} and $\cG = N(H)/H$ finite we see that $\cM \times G/H \to \rM =
M$ is a finite cover. Since $M$ is simply connected this is a trivial
$\cG$-bundle, i.e., $\cM \simeq \cG \times N$, where $N$ is a connected
component of $\cM$, and $M \simeq (\cG \times N) \times _{\cG} G/H \simeq N
\times G/H$ as $G$-manifolds.
\end{proof}

Let us now turn to metric properties of the core resolution. - If $M$ is a
Riemannian $G$-manifold, it is natural to equip the totally geodesic core $\cM$
with the induced Riemannian metric. It is also natural to endow $G/H$ with a
metric induced from a biinvariant on $G$. In this case, the product metric on
$\cM \times G/H$ will be invariant under both the $G$-action and the
$\cG$-action. By the Gray-O'Neill curvature submersion formula (cf. \cite{ON} or
\cite{G}), we get the following interesting fact:

\begin{proposition}
\label{2.10}
The resolution $\rM$ of any Riemannian $G$-manifold $M$ with sectional
curvature, sec$M \ge k, k \le 0$ supports a $G$-invariant Riemannian metric with
sec$\rM \ge k$.
\end{proposition}

\begin{remark}
\label{2.11}
This is particularly interesting for $k = 0$. For example, this provides a new
proof that e.g. $CP^n \# -CP^n$ has a metric with non-negative curvature,
first proved in \cite{C}. The point is that $CP^n \# -CP^n$ is the resolution,
${}^rS^{2n}$ of the $U(n)$-manifold $S^{2n}$, where the $U(n)$-action is the
suspension of the standard action on $S^{2n-1} = U(n)/U(n-1)$.
\end{remark}

From Corollary \ref{2.8} and the description $\cM \to \rM = \cM \times _{\cG}G/H \to
G/N(H)$ it is tempting to make the following conjecture.

\begin{conjecture}
\label{2.12}
Let $M$ be a positively curved $G$-manifold with principal isotropy group $H \ne
\{1\}$. Then either

\no $(a)$ $M = G/H$ , or

\no $(b)$ $\rM \ne M$. In particular there are singular $G$-orbits in $M$.
\end{conjecture}

Note that the assumption $H \ne \{1\}$ is necessary since there are many linear
almost free actions of $S^1$ and of $S^3$ on spheres (cf. e.g. \cite{GG}). These
together with the transitive actions on spheres all have ${}^rS = S$. Note that if the
above conjecture is correct, then any (non finite) isometric $G$-action on a positively
curved manifold, $M$ has singular orbits, unless the action is transitive, or $G$
acts almost freely on $M$, in which case rk$G$ = 1 and dim$M$ is odd.

Another interesting question about the resolution construction is whether it
preserves the topological dichotomy into elliptic and hyperbolic types \cite{FHT}.
From its very construction this hinges on the question of whether the core can
change type or not.

\begin{problem}
\label{2.13}
Is $\rM$, or equivalently $\cM$ elliptic if and only if $M$ is ?
\end{problem}

In conclusion we also point out that "intermediate" resolutions can be
constructed by replacing the core by other reductions.

\providecommand{\bysame}{\leavevmode\hbox to3em{\hrulefill}\thinspace}

\end{document}